\documentclass[12pt]{amsart}

\usepackage{amssymb}
\usepackage{bm}
\usepackage{graphicx}
\usepackage{psfrag}
\usepackage[usenames,dvipsnames]{xcolor}
\usepackage{enumerate}
\usepackage{url}

\usepackage{algpseudocode}

\usepackage{mathtools}

\usepackage{subcaption,framed}
\newtheorem{theorem}{Theorem}[section]
\newtheorem{corollary}[theorem]{Corollary}

\theoremstyle{definition}

\DeclareMathOperator\Ap{Ap} 
\renewcommand\>{\right\rangle}
\newcommand\<{\left\langle}

\newcommand\ZZ{\mathbb{Z}}
\newcommand\QQ{\mathbb{Q}}

\begin{document}

\mbox{}
\title{Elasticity in Ap\'ery sets}

\author[J.~Autry]{Jackson Autry}
\address{Mathematics and Statistics Department\\San Diego State University\\San Diego, CA 92182}
\email{jautry@sdsu.edu}

\author[T.~Gomes]{Tara Gomes}
\address{Mathematics and Statistics Department\\San Diego State University\\San Diego, CA 92182}
\email{gomes.tara@gmail.com}

\author[C.~O'Neill]{Christopher O'Neill}
\address{Mathematics and Statistics Department\\San Diego State University\\San Diego, CA 92182}
\email{cdoneill@sdsu.edu}

\author[V.~Ponomarenko]{Vadim Ponomarenko}
\address{Mathematics and Statistics Department\\San Diego State University\\San Diego, CA 92182}
\email{vponomarenko@sdsu.edu}

\date{\today}

\begin{abstract}
A numerical semigroup $S$ is an additive subsemigroup of the non-negative integers, containing zero, with finite complement. Its multiplicity $m$ is its smallest nonzero element.  The Ap\'ery set of $S$ is the set $\Ap(S) = \{n \in S : n-m \notin S\}$. Fixing a numerical semigroup, we ask how many elements of its Ap\'ery set have nonunique factorization, and define several new invariants.
\end{abstract}

\maketitle

\section{Introduction}
\label{sec:intro}

Every child's first semigroup is the natural numbers, and their first factorization theorem is the Fundamental Theorem of Arithmetic, which gives unique factorization as a product of primes.  The other operation, addition, is not addressed.  Much attention has been given to factorization in various semigroups; for a general introduction, see~\cite{MR2194494}.  Often, the operation is multiplication \cite{MR2869518,MR3941724,MR3035125}, but addition is worth studying as well~\cite{MR556658}; it will be our operation here.  

A \emph{numerical semigroup} $S$ is a subset of $\ZZ_{\ge 0}$ with finite complement that is closed under $+$ and contains $0$.  Numerical semigroups have been the subject of considerable recent study \cite{MR3885968,Hilbert,MR3722042,MR3503387,MR3295662,MR3324682}.  Many applications are known, such as in coding theory~\cite{MR3288285}.  For a general introduction to numerical semigroups, see~\cite{MR3558713} or~\cite{MR2549780}.  


The atoms of a numerical semigroup $S$ are the nonzero elements that cannot be expressed as the combination of two nonzero elements.  The set $\mathcal A(S)$ of atoms of $S$ is finite; we call $e(S) = |\mathcal A(S)|$ the \emph{embedding dimension} of $S$.   We write $\<a_1, a_2,\ldots, a_k\>$, with $a_i$ listed in ascending order, to denote the numerical semigroup with atoms $a_1, \ldots, a_k$.  The smallest atom $a_1$ is also the smallest nonzero element of $S$; we call it $m(S)$,  the \emph{multiplicity} of $S$.  

An important tool for the study of numerical semigroups, from \cite{MR17942}, is the Ap\'ery set
$$\Ap(S) = \{x \in S : x - m(S) \notin S\},$$
which contains the smallest element of $S$ in each congruence class modulo $m(S)$.  It~is easy to show that $|\Ap(S)| = m(S)$ and $\mathcal A(S) \setminus \{m(S)\} \subseteq \Ap(S)$.  
If we want to express elements of $S$ as a free combination of atoms, $\mathcal A(S)$ is what we study.  However, if we want to use as many copies of $m(S)$ as possible and other atoms as little as possible, we look to $\Ap(S)$ instead.  

We study properties about factorization into atoms. The most famous factorization invariant is elasticity.  Given a semigroup $S$ and some $x\in S$, we write $x$ as the combination of atoms in every possible way.  The \emph{elasticity} of $x$, denoted $\rho(x)$, is the the largest number of atoms that can be used, divided by the smallest number.   Clearly $\rho(x) \ge 1$; if equality holds, we say $x$ is \emph{half-factorial}.  Conventionally we say that the unit $0$ is half-factorial.

In this note, we consider the elasticity function restricted to elements of $\Ap(S)$.  We write $\rho(\Ap(S))$ for the maximum elasticity over the elements of $\Ap(S)$; if $\rho(\Ap(S)) = 1$, we say that $S$ is \emph{Ap\'ery half factorial}, or AHF.  

We can visualize the factorization structure of $\Ap(S)$ using a partially ordered set $(\Ap(S), \preceq)$ with $x \preceq y$ whenever $y - x \in S$, called the \emph{Ap\'ery poset} of $S$.  The Hasse diagrams of two Ap\'ery posets are depicted in Figure~\ref{f:posets}.  The atoms of the Ap\'ery poset (i.e., the elements directly above the unique minimal element $0$) are precisely the elements of $\mathcal A(S)$ apart from $m(S)$, and an edge connects $x$ up to $y$ in the Hasse diagram exactly when $y - x \in \mathcal A(S)$.  
This leads to the following interesting observation.  

\begin{theorem}\label{t:lengthgrading}
A numerical semigroup $S$ has graded Ap\'ery poset if and only if $S$ is~AHF.  
\end{theorem}

\begin{proof}
By the above discussion, each length $\ell$ chain (set of mutually comparable elements) from $0$ to an element $n \in S$ corresponds to an ordered factorization of~$n$ with length $\ell$.  As such, two different chain lengths are present if and only if $n$ is not half-factorial.  
\end{proof}

\begin{figure}[t!]
\begin{center}
\begin{subfigure}[t]{.38\linewidth}
\begin{center}
\includegraphics[width=1.5in]{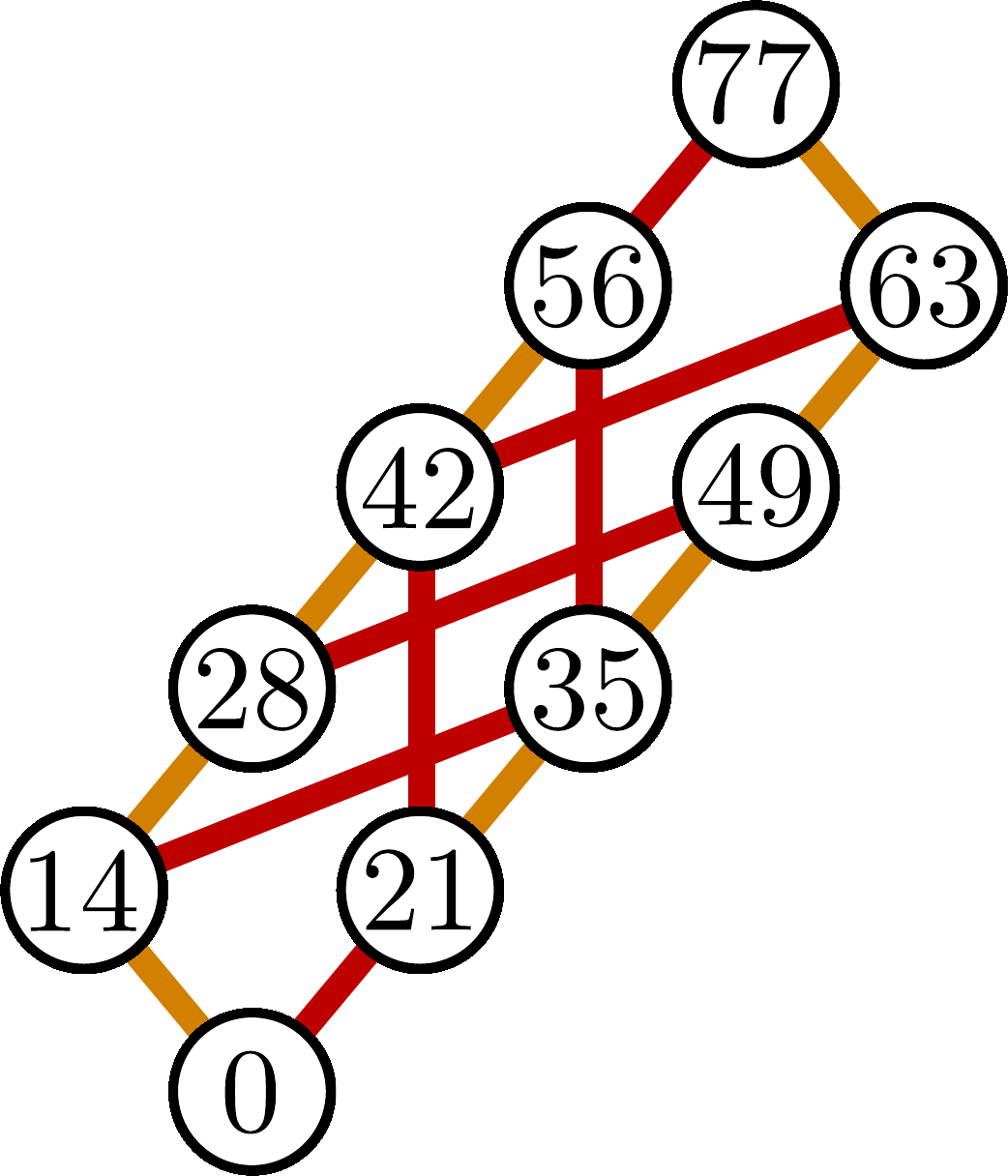}
\end{center}
\caption{Poset for $S = \<10,14,21\>$, as in Theorem~\ref{t:smallahff} with $n = 2$ and $p = 7$.}
\end{subfigure}
\qquad\qquad
\begin{subfigure}[t]{.38\linewidth}
\begin{center}
\includegraphics[width=1.5in]{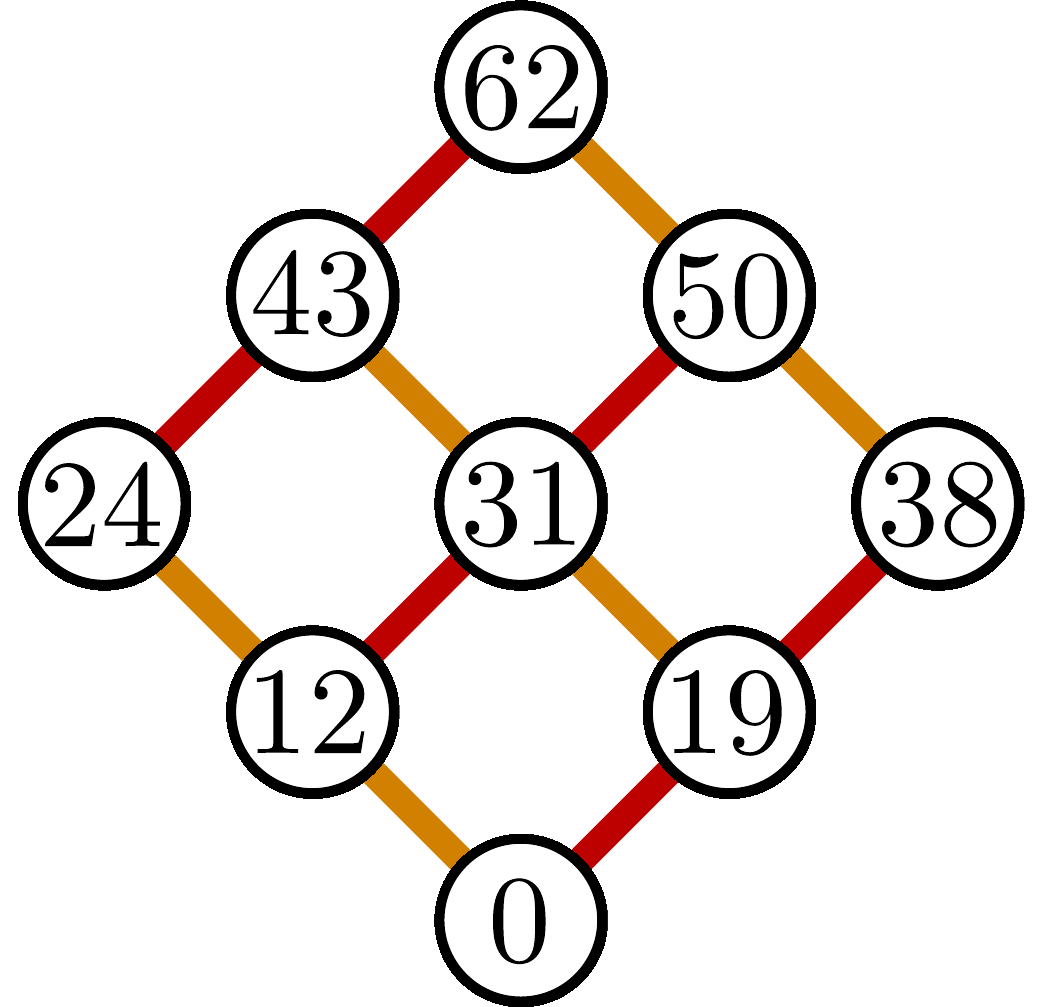}
\end{center}
\caption{Poset for $S = \<9,12,19\>$, as in Theorem~\ref{t:ahfexamples} with $n = 3$.}
\end{subfigure}
\end{center}
\caption{Examples of Ap\'ery posets.}
\label{f:posets}
\end{figure}

\section{Ap\'ery elasticities}
\label{sec:maxaperyelast}

We begin by observing that if $e(S) = 2$, then we can write $S = \<m, a\>$ and 
$$\Ap(S) = \{0, a, 2a, \ldots, (m-1)a\}.$$
Each element of $\Ap(S)$ is then not only half-factorial, but has unique factorization.  On~the other hand, if $S$ has \emph{maximal embedding dimension} (i.e., $e(S) = m(S)$), then $S = \<m, a_1, a_2, \ldots, a_{m-1}\>$ and $\Ap(S) = \{0, a_1, a_2, \ldots, a_{m-1}\}$.  Again each element of $\Ap(S)$ has unique factorization.  These observations are extended slightly as follows.

\begin{theorem}\label{t:extremenumgens}
Let $S$ be a numerical semigroup with $e(S) = 2$ or $e(S) \ge m(S) - 1$.  Then  $S$ is Ap\'ery half-factorial.
\end{theorem}

\begin{proof}
If $e(S) = m(S) - 1$, then the only element of $\Ap(S)$ that is not an atom only has length 2 factorizations.  
\end{proof}

Theorem \ref{t:extremenumgens} can't be extended in general to smaller embedding dimension than $m(S)-1$.  Consider $S = \<5, 6, 9\>$, where $m = 5$, $a_1 = 6, a_2 = 9$.    Now 
$$\Ap(S)=\{0,a_1,a_2, 2a_1, 3a_1=2a_2\},$$
so $\rho(3a_1) = \frac{3}{2}$.  

Given a subset $T \subset S$, define the \emph{set of elasticities} of $T$ as
$$R(T) = \{\rho(n) : n \in T\}.$$
This invariant has been studied for numerical semigroups in \cite{MR3602830}, wherein $R(S)$ is characterized for all but finitely many elasticities coming from ``small'' elements of~$S$.  As such, $R(\Ap(S))$ is a natural starting place for studying the remainder of $R(S)$.

It is easy to see that if $e(S) = m(S) - 2$, either $S$ is AHF or $R(\Ap(S)) = \{1,\frac{3}{2}\}$.  Determining $R(\Ap(S))$ for other near-maximal embedding dimensions remains open.

We now present a family of semigroups in Theorem~\ref{t:masterexample} which demonstrate several extremal behaviors, as discussed thereafter.  

\begin{theorem}\label{t:masterexample}
Fix $a > b \ge 1$ with $\gcd(a,b) = 1$.  There is a numerical semigroup $S$ with (i) $R(\Ap(S)) = \{1, \tfrac{a}{b}\}$ and (ii) only one element of $\Ap(S)$ has elasticity $\tfrac{a}{b}$.  
\end{theorem}

\begin{proof}
Fix a prime $p \nmid (a + b)$ with $a + b < pb$, and let $S = \<a + b, pa, pb\>$.  
We have
$$\Ap(S) = \{0, pb, 2pb, \ldots, (a - 1)pb, pa, 2pa, \ldots, (b - 1)pa, pab\},$$
wherein each element has unique factorization except $pab$, which has elasticity $\tfrac{a}{b}$.
\end{proof}

One natural question to ask is:\ which subsets of $\QQ_{\ge 1}$ can occur as $R(\Ap(S))$ for some numerical semigroup $S$?  Certainly we must have $1 \in R(\Ap(S))$, and the sole singleton subset, $\{1\}$, is achieved for all Ap\'ery half-factorial $S$.  All subsets of size two are realizable by Corollary~\ref{c:2elasticities}.  Larger subsets of $\QQ_{\ge 1}$ remain unresolved.  

\begin{corollary}\label{c:2elasticities}
Given $r \in \QQ_{>1}$, some numerical semigroup $S$ has $R(\Ap(S)) = \{1, r\}$.  
\end{corollary}

\begin{proof}
Write $r = \frac{a}{b}$ in reduced form, and apply Theorem~\ref{t:masterexample}.  
\end{proof}

Since $\Ap(S)$ is a finite set, we can consider the full distribution of elasticity over its elements, and not just its maximum $\rho(\Ap(S))$.  We call the \emph{Ap\'ery half-factorial fraction}, or AHFF, the ratio of the number of half-factorial elements of $\Ap(S)$, to $|\Ap(S)|$.  If $S$ is AHF, then its AHFF is $1$.

Theorem~\ref{t:masterexample} produced a single non-half-factorial element of $\Ap(S)$; hence  $S$ had AHFF close to $1$.    Certainly the AHFF cannot be zero, as each element of $\mathcal A(S)$ is half-factorial. One wonders how small the AHFF can be.   Theorem~\ref{t:smallahff}, illustrated in Figure~\ref{f:posets}(a), displays the smallest possible AHFF while maintaining $e(S) = 3$. 

\begin{theorem}\label{t:smallahff}
The fraction of Ap\'ery set elements of a numerical semigroup that are half-factorial can be arbitrarily close to $0$.
\end{theorem}

\begin{proof}
Let $p, n \in \ZZ_{\ge 1}$ with $p$ prime, $p \neq 5$, and $2p > 5n > 5$.  Set $m = 5n$, $a_1 = 2p$, $a_2 = 3p$, and take $S = \<m, a_1, a_2\>$.  We have 
$$\Ap(S) = \{0, 2p, 3p, \ldots, (5n - 1)p, (5n + 1)p\}.$$
Since $6p = 3a_1 = 2a_2$, only $0$, $2p$, $3p$, $4p$, $5p$ and $7p$ are half-factorial in $\Ap(S)$.  As~such, the AHFF of $S$ is $\frac{6}{5n}$.
\end{proof}

With the generality of the family in Theorem~\ref{t:masterexample}, one might wonder if any $S$ with $e(S) = 3$ can be AHF.  One such family is provided in Theorem~\ref{t:ahfexamples}, an example of which is illustrated in Figure~\ref{f:posets}(b).  

\begin{theorem}\label{t:ahfexamples}
For each $n \in \ZZ_{\ge 2}$, the semigroup $S = \<n^2, n^2 + n, 2n^2 + 1\>$ is AHF.
\end{theorem}

\begin{proof}
$\Ap(S) = \{a(n^2 + n) + b(2n^2 + 1) : 0 \le a,b \le n - 1\}$.
\end{proof}

Theorem~\ref{t:ahfexamples} also demonstrates that the width of the Ap\'ery poset, which is always bounded below by $e(S)$, can be larger. 

\begin{corollary}\label{c:largewidth}
The width of an Ap\'ery poset can be arbitrarily large, even for $e(S) = 3$.  
\end{corollary}

\section{Mean Ap\'ery elasticity}
\label{sec:meanaperyelast}

Motivated in part by recent investigations into ``average'' factorization lengths in numerical semigroups \cite{MR3921068}, we next consider the \emph{mean Ap\'ery elasticity}, i.e., 
$$MAE(S) = \frac{1}{|\Ap(S)|} \sum_{n \in \Ap(S)} \rho(n).$$
If $S$ is half-factorial, of course $MAE(S) = 1$.  The family from Corollary~\ref{c:2elasticities} has mean Ap\'ery elasticity $1 + \frac{1}{b} - \frac{2}{a + b}$.  Theorem~\ref{t:largemae} will show that mean Ap\'ery elasticity may be arbitrarily large, though one may still wonder which elements of $\QQ_{\ge 1}$ occur as $MAE(S)$ for some numerical semigroup $S$.

\begin{theorem}\label{t:largemae}
The values of $MAE(S)$, with $e(S) = 3$, can be arbitrarily large.
\end{theorem}

\begin{proof}
Let $p, q$ be odd primes with $p > 2q + 4$.  Set $m = 4q + 8, a_1 = 2p, a_2 = qp$, and take $S=\<m,a_1,a_2\>$.  We have 
$$Ap(S) = \{0, 2p, 4p, \ldots, (q - 1)p, qp, (q + 1)p, \ldots, \tfrac{1}{2}(9q + 17)p\},$$
where all multiples of $p$ are present after $qp$ except $(4q + 8)p$.  Now, consider the set
$$T = \{(2q + 2i)p : 0 \le i < q\} \subset \Ap(S).$$
We calculate elasticity of the elements of $T$ as
$$\rho((2q + 2i)p) = \rho((q + i)a_1) = \rho(2a_2 + ia_1) = \frac{q + i}{2 + i} \ge \frac{q}{2 + i},$$
and consequently 
$$\begin{array}{r@{}c@{}l}
MAE(S)
&{}={}& \displaystyle \frac{1}{m} \sum_{n \in \Ap(S)} \rho(n)
\ge \frac{3q + 8}{m} + \frac{1}{m} \sum_{n \in T} \rho(n)
\ge \frac{3q + 8}{m} + \frac{q}{m} \sum_{i = 0}^{q - 1} \frac{1}{2 + i} \\
\end{array}$$
grows arbitrarily large as $q \to \infty$.
\end{proof}

\section{Asymptotic distributions}
\label{sec:asymptotic}

Given a numerical semigroup $S$, denote by $g(S) = |\ZZ_{\ge 0} \setminus S|$ the \emph{genus} of $S$.  Let $n_g$ denote the number of numerical semigroups with genus $g$, and let $n_{m,g}$ denote the number of numerical semigroups with multiplicity $m$ and genus $g$.  For example, letting $f_g$ denote the $g$'th Fibonacci number, it was recently proven that $n_g/f_g$ approaches a constant as $g \to \infty$~\cite{MR3053785}, although it is still open whether $n_{g+1} \ge n_g$ for every $g \ge 0$.  On the other hand, for fixed $m$, the ratio $n_{m,g} / g^{m-1}$ approaches a constant as $g \to \infty$.  

There has been a recent push to understand the distribution of numerical semigroups with a given genus across different special families.  For example, if $M_g$ and $M_{m,g}$ denote, respectively, the number of maximal embedding dimension numerical semigroups with genus $g$ and the number with both multiplicity $m$ and genus $g$, then $M_g/n_g \to 0$ as $g \to \infty$, while $M_{m,g}/n_{m,g} \to 1$ as $g \to \infty$; see~\cite{MR3978504,MR2875324}.  

Continuing in this vein, let $h_g$ denote the number of AHF numerical semigroups with genus $g$, and let $h_{m,g}$ denote the number of AHF numerical semigroups with multiplicity $m$ and genus $g$.  Theorems~\ref{t:asympfixedm} and~\ref{t:asympgenus} below demonstrate that AHF numerical semigroups form a much larger class than those with maximum embedding dimension.  

\begin{theorem}\label{t:asympfixedm}
For each fixed $m \ge 2$, we have
$$\lim_{g \to \infty} \frac{h_{m,g}}{n_{m,g}} = 1.$$
\end{theorem}

\begin{proof}
Apply \cite[Corollary~1]{MR3978504}.  
\end{proof}

Identifying the precise value of the limit below will likely be challenging, considering the long and technical nature of the proof of \cite[Theorem~1]{MR3053785}.  Out of the 1179593 numerical semigroups with genus at most 25, we find 1032971 (about 88\%) are AHF.

\begin{theorem}\label{t:asympgenus}
We have
$$0 < \lim_{g \to \infty} \frac{h_g}{n_g} < 1.$$
\end{theorem}

\begin{proof}
Let $f_n$ denote the $n$'th Fibonacci number.  By \cite[Theorem~1]{MR3053785}, we have
$$\lim_{g \to \infty} \frac{f_{g+1}}{n_g} > 0.$$
As such, for the first inequality, it suffices to show that $f_{g+1} \le h_g$.  Fix a multiplicity $m \le g + 1$.  For each subset $T \subset \{1, \ldots, m-1\}$, consider the numerical semigroup $S$ with Ap\'ery set given by $\Ap(S) = \{0, a_1, \ldots, a_{m-1}\}$, where
$$a_i = \begin{cases}
2m + i & \text{if } i \in T; \\
m + i & \text{if } i \notin T.
\end{cases}$$
It is clear $S$ has multiplicity $m$ and genus $m + |T|$, and is AHF.  As such,
$$h_g = \sum_{m = 2}^{g+1} h_{m,g} \ge \sum_{m = 2}^{g+1} \binom{m - 1}{g - (m - 1)} = f_{g+1}.$$
For the other inequality, we use a similar construction, where we first let $a_1 = m + 1$, $a_2 = 2m + 2$, $a_3 = 3m + 3$, $a_4 = m + 4$, and $a_{m-1} = m + (m - 1)$, and then choose the remaining $a_i$ as above.  In each resulting semigroup, 
$$a_3 = 3a_1 = a_4 + a_{m-1}$$
is not half-factorial, and by similar reasoning to above, this family of semigroups also comprises a positive asymptotic proportion of those with genus $g$.  
\end{proof}


\newcommand{\noop}[1]{}

\end{document}